\documentclass[11pt]{amsart}

\usepackage{amssymb, graphicx}
\usepackage{bbm,titletoc} 
\usepackage[normalem]{ulem}

\renewcommand{\subset}{\subseteq}
\renewcommand{\emptyset}{\varnothing}

\usepackage{enumerate}
\usepackage[usenames,dvipsnames]{color}

\makeatletter
\def\moverlay{\mathpalette\mov@rlay}
\def\mov@rlay#1#2{\leavevmode\vtop{%
   \baselineskip\z@skip \lineskiplimit-\maxdimen
   \ialign{\hfil$#1##$\hfil\cr#2\crcr}}}

\makeatother

\numberwithin{equation}{section}

\definecolor{Ruta}{rgb}{0.309, 0.459, 0.208}
\definecolor{Vino}{rgb}{0.256,0,0}
\definecolor{Siva}{rgb}{0.116,0.116,0.116}
\definecolor{Siva'}{rgb}{0.250,0.116,0.116}

\usepackage{amsfonts}
\usepackage{amsmath}
\usepackage{verbatim}

\usepackage[pagebackref,colorlinks,linkcolor=red,citecolor=blue,urlcolor=blue,hypertexnames=true]{hyperref}

\usepackage{blkarray}
\usepackage{multirow}
\usepackage{xspace} 
\usepackage{multirow,bigdelim}

\newcommand{\cc}{^{\bf{c}}}
\newcommand{\rr}{^{\bf{r}}}

\def\beq{ \begin{equation} }
\def\eeq{ \end{equation} }

\def\bep{\begin{proof}}
\def\eep{\end{proof}}

\def\ben{ \begin{enumerate} }
\def\een{ \end{enumerate} }

\newcommand{\id}{{\rm id}}

\newcommand{\tnz}{\otimes}
\newcommand{\ol}{\overline}
\newcommand{\tr}{\mathrm{tr}}

\allowdisplaybreaks[1]

\newtheorem{theorem}{Theorem}[section]
\newtheorem{proposition}[theorem]{Proposition}

\newtheorem{corollary}[theorem]{Corollary}

\theoremstyle{definition}

\newtheorem{lemma}[theorem]{Lemma}

\newtheorem{remark}[theorem]{Remark}

\newtheorem{example}[theorem]{Example}

\newcommand{\X}{\langle x\rangle}

\newcommand{\s}{\sigma}

\newcommand{\N}{\mathbb{N}}

\newcommand{\F}{\mathbb{F}}

\newcommand{\g}{m}
\newcommand{\td}{\tilde}

\newcommand{\m}{s}  
\newcommand{\n}{r}

\newcommand{\inn}{{\rm in}}
\newcommand{\sym}{{\rm Sym\,}}
\newcommand{\com}{\mathsf{c}}

\newcommand{\st}{\;\ifnum\currentgrouptype=16 \middle\fi|\;}

\usepackage{graphicx}
\usepackage{calc}
\newlength{\depthofsumsign}
\makeatletter
\newcommand*{\rom}[1]{\expandafter\@slowromancap\romannumeral #1@}
\makeatother

\begin{document}

\title[Functional identities on matrix algebras]{Functional identities on matrix algebras}

\author{Matej Bre\v sar}
\author{\v Spela \v Spenko}
\address{M. Bre\v sar,  Faculty of Mathematics and Physics,  University of Ljubljana,
 and Faculty of Natural Sciences and Mathematics, University
of Maribor, Slovenia} \email{matej.bresar@fmf.uni-lj.si}
\address{\v S. \v Spenko,  Institute of  Mathematics, Physics, and Mechanics,  Ljubljana, Slovenia} \email{spela.spenko@imfm.si}

\begin{abstract}
Complete solutions of functional identities  $\sum_{k\in K}F_k(\overline{x}_m^k)x_k = \sum_{l\in L}x_lG_l(\overline{x}_m^l)$ on the matrix algebra $M_n(\F)$ are given. The nonstandard parts of these solutions turn out to follow from the Cayley-Hamilton identity.
\end{abstract}

\thanks{2010 {\em Math. Subj. Class.} Primary 16R60. Secondary 13D02, 13P10,
 16R50.}
\thanks{Supported by ARRS Grant P1-0288.}

\maketitle

\section{Introduction}

A functional identity on a ring $R$ is an identical relation that involves arbitrary elements from $R$ along with functions which are considered as unknowns. 
The fundamental example of such an identity is
\begin{equation}\label{e1}
\sum_{k\in K}F_k(\overline{x}_m^k)x_k = \sum_{l\in L}x_lG_l(\overline{x}_m^l)\quad\mbox{for all $x_1\dots,x_m\in R$},
\end{equation}
where $\overline{x}_m^k = (x_1,\dots,\hat x_{k},\dots,x_m)$, $K$ and $L$ are subsets of $\{1,\dots,m\}$, and $F_k,G_l$ are arbitrary functions from $R^{m-1}$ to $R$. The description of these functions, which is the usual  goal when facing \eqref{e1}, has turned to be applicable to various mathematical problems. We refer the reader to  \cite{FIbook} for an account of the theory of functional identities and its applications.

 Up until very recently functional identities have been studied only in rings in which \eqref{e1} has only so-called {\em standard solutions} (see  Subsection \ref{sub41} for the definition). This excludes the  basic case where $R=M_n(\mathbb F)$, the algebra of $n\times n$ matrices over a field $\mathbb F$, unless   $|K|\le n$ and $|L|\le n$. 
 Various applications of  the general theory of functional identities
 therefore do not cover $M_n(\mathbb F)$ with small $n$, although  they yield definitive results for large classes of algebras whose dimension is either infinite or finite but big enough.
 
 The recent papers \cite{BS14} and \cite{BPS14}  give, to the best of our knowledge, the first complete results on functional identities of $M_n(\mathbb F)$.  The first one considers functional identities in one variable, and the second one  considers quasi-identities; i.e., identities of the type \eqref{e1} with $L=\emptyset$ and each $F_k$ being a sum of  scalar-valued functions multiplied by  noncommutative monomials. In the present paper we take a step further and consider the general functional identity \eqref{e1}. We will give a full description of the functions $F_k$ and $G_l$ under the natural assumption that they are   multilinear. Thus, in some sense we can say now that functional identities are finally understood in both finite and infinite dimensional context.
 
The Cayley-Hamilton theorem yields the fundamental example of a functional identity of $M_n(\mathbb F)$ which does not have only standard solutions. We call it the {\em Cayley-Hamilton identity}. From the results in \cite{BS14} it is evident that all nonstandard solutions of functional identities in one variable of $M_n(\mathbb F)$ follow from the Cayley-Hamilton identity, while the main result of \cite{BPS14} shows that there exist quasi-identities of $M_n(\mathbb F)$ which are not a consequence of the Cayley-Hamilton identity; in more technical terms, the T-ideal of all  quasi-identities is not generated by the Cayley-Hamilton identity. The results in this paper will show that  the nonstandard parts of the solutions of \eqref{e1} follow from the  Cayley-Hamilton identity. (This does not contradict the result of \cite{BPS14} since the T-ideal of quasi-identities is a proper subset of the set of identities treated here.)  Our main idea of handling \eqref{e1} is to  interpret this functional  identity 
 coordinate-wise and then apply the theory of syzygies  (in fact, functional identities may be viewed as a kind of ``noncommutative syzygies").

The paper is organized as follows. In Section \ref{secgpi} we give some remarks on generalized polynomial identities, which are, in the context of matrix algebras,  more general than the functional identities \eqref{e1}. Our main results are obtained in Sections  \ref{stri} and  \ref{s4}.
In Section  \ref{stri} we describe all solutions of  the
 one-sided functional identities, i.e., those with  $K=\emptyset$ or  $L=\emptyset$, 
by applying the result on generators of syzygies on generic matrices from   \cite{On}. 
In  Section \ref{s4}, which is basically independent of Section \ref{stri}, we show  that every solution of the two-sided functional identity \eqref{e1} is {\em standard modulo one-sided identities}.  Roughly speaking, this means that it can be expressed by standard solutions and solutions of one-sided identities.
A precise definition is given in Subsection \ref{sub41}. We end the paper with an application to the  theory of commuting functions.

\section{Generalized polynomial identities}\label{secgpi}

It is well-known that the T-ideal of trace identities of $M_n(\mathbb F)$ is generated, as a T-ideal, by the Cayley-Hamilton identity \cite[p. 444]{Proc07}. In this sense every polynomial identity is a consequence of the Cayley-Hamilton identity. We have addressed ourselves the question whether a similar statement holds for the generalized polynomial identities. As we shall see, the answer is positive in the  char$(\mathbb F) = 0$ case, and the proof is not difficult. From the nature of this result one would expect that it should be known, but we were unable to find it in the literature.

Let us recall the necessary definitions on generalized polynomial identities (for more details see \cite{BMM}).
  Let $x = \{x_1,x_2,\dots\}$ be a set of noncommuting indeterminates, and let $\mathbb F\langle x\rangle$ be the free algebra on $X$. Consider the free product  $M_n(\F)\ast\F\X$ over $\F$. Its elements are sometimes called {\em generalized polynomials}. Informally they can be viewed as sums of expressions of the form $a_{i_0}x_{j_1}a_{i_1}\dots a_{i_{k-1}}x_{j_k}a_{i_k}$ where $a_{i_\ell}\in M_n(\F)$. Given $f=f(x_1,\dots,x_k)\in M_n(\F)\ast\F\X$ and $b_1,\dots,b_k\in M_n(\F)$,
we define the {\em evaluation} $f(b_1,\dots,b_k)$ in the obvious way. We say that $f$ is a {\em generalized polynomial identity} (GPI) of $ M_n(\F)$ if $f(b_1,\dots,b_k) =0$ for 
all $b_i\in M_n(\F)$. For example, if $e$ is a rank one idempotent in $M_n(\F)$, then $[ex_1e,ex_2e]$ is readily a GPI.
 We remark that 
every identity of the form \eqref{e1} can be interpreted as a GPI. This is because every multilinear function $F:M_n(\F)^{m-1}\to M_n(\F)$ is equal to a sum of functions of the form
$(b_1,\dots,b_{m-1})\mapsto a_{i_0}b_{j_1}a_{i_1}\dots a_{i_{m-2}}b_{j_{m-1}}a_{i_{m-1}}$ where $\{j_1,\dots,j_{m-1}\} = \{1,\dots,m-1\}$. 
 
An ideal $I$ of $M_n(\F)\ast\F\X$ is said to be a {\em T-ideal}
if $f(x_1,\dots,x_k)\in I$ and $g_1,\dots,g_k\in M_n(\F)\ast\F\X$ implies $f(g_1,\dots,g_k)\in I$. The set of all GPI's of $M_n(\F)$ is clearly a T-ideal. The next proposition was obtained, in some form, already in \cite{Lit31}, and later extended to considerably more general rings by Beidar (see e.g.~ \cite{BMM}). We will give a short alternative proof. Let us first recall the following elementary fact: If  an algebra $B$ contains a set of $n\times n$ matrix units; i.e., a set of elements $e_{ij}$, $1\leq i,j\leq n$, satisfying  
$$e_{ij}e_{kl}=\delta_{jk}e_{il},\quad \sum_{i=1}^n e_{ii}=1,$$
then $B$ is isomorphic to the matrix algebra $M_n(A)$ where $A=e_{11}Be_{11}$.

\begin{proposition}[{\cite[Proposition 6.5.5]{BMM}}]\label{gpi}
Let $e$ be a rank one idempotent in $M_n(\F)$.
\begin{enumerate}
\item If $|\F|=\infty$, then the T-ideal of all GPI's of $M_n(\F)$ is generated by $[ex_1e,ex_2e]$.
\item  If $|\F|=q$,  then the T-ideal of all GPI's of $M_n(\F)$ is generated by $[ex_1e,ex_2e]$ and $(ex_1e)^q-ex_1e$.
\end{enumerate}
\end{proposition}
\begin{proof}
Pick a set of matrix units $e_{ij}$, $1\leq i,j\leq n$, of $M_n(\F)$ so that $e_{11}=e$. 
As $M_n(\F)\ast\F\X$ contains  $M_n(\F)$ as a subalgebra, it also contains all $e_{ij}$. Accordingly, $M_n(\F)\ast\F\X\cong M_n(A)$ where  $A = e_{11}\bigl(M_n(\F)\ast\F\X\bigr)e_{11}$. Since  $M_n(\F)\ast\F\X$  is generated 
by the elements $$e_{si}x_k e_{jt} = e_{s1}(e_{1i}x_k e_{j1})e_{1t},$$ it follows that $A$ is generated by the elements 
$$x_{ij}^{(k)}:= e_{1i}x_k e_{j1},\quad 1\le i,j \le n,\,\,\, k = 1,2,\dots$$ 
Note that $A$ is actually the free algebra on the set $\overline{x}:=\{x_{ij}^{(k)}\,|\,1\le i,j \le n,\,\, k = 1,2,\dots\}$.

Let ${\rm Hom}_{M_n}(M_n(\F)\ast \F\X,M_n(\F))$ denote the set of algebra homomorphisms from $M_n(\F)\ast \F\X$ to $M_n(\F)$ that act as the identity on $M_n(\F)$. Identifying
$M_n(\F)\ast \F\X$ with $M_n(A)$ one easily sees
 that there is a canonical isomorphism 
$${\rm Hom}_{M_n}(M_n(\F)\ast \F\X,M_n(\F))\cong  {\rm Hom}(A,\F).$$
Now take a GPI $f$ of $M_n(\F)$. This means 
that $$f\in \cap_{\phi\in {\rm Hom}_{M_n}(M_n(\F)\ast \F\X,M_n(\F))}{\rm ker}\phi,$$ or equivalently,   $$ e_{1 i}fe_{j1}\in  
\cap_{\phi\in {\rm Hom}(A,\F)}\ker \phi$$ for every $1\leq i,j\leq n$. 
Since $A$ is the free algebra on $\overline{x}$ it can be easily shown  that $\cap_{\phi\in {\rm Hom}(A,\F)}\ker \phi$ is generated by $$[x_{ij}^{(k)},x_{pq}^{(l)}] = [ex_{ij}^{(k)}e,ex_{pq}^{(l)}e]$$ and 
additionally by $$\Bigl(x_{ij}^{(k)}\Bigr)^q - x_{ij}^{(k)} = \Bigl(ex_{ij}^{(k)}e\Bigr)^q - ex_{ij}^{(k)}e $$
 if $|\F|=q$. Hence $f= \sum_{i,j} e_{i1}(e_{1i}fe_{j1})e_{1j}$ lies in the T-ideal generated by $[ex_1e,ex_2e]$, and additionally by $(ex_1e)^q-ex_1e$ if  $|\F|=q$. 
\end{proof}

Let   
$$
q_n(x_1) = x_1^n  +  \tau_1(x_1)x_1^{n-1} + \cdots + \tau_n(x_1)
$$
denote  the Cayley-Hamilton polynomial. Thus, $\tau(x_1) = -{\rm tr}(x_1)$ and $\tau_n(x_1)=(-1)^n \det(x_1)$. Until the end of this section we assume that  char$(\F)=0$. Then each $\tau_i(x_1)$ can be expressed   as a $\mathbb{Q}$-linear combination of the products of  ${\rm tr}(x_1^j)$.  Let  
$$Q_n=Q_n(x_1,\ldots,x_n)$$ denote the multilinear version of $q_n(x_1)$ obtained by full polarization.
  For example, $$Q_2(x_1,x_2) = x_1x_2 + x_2x_1 -\tr(x_1)x_2 - \tr(x_2)x_1 + \tr(x_1)\tr(x_2)- \tr(x_1x_2) .$$
  As $Q_n(b_1,\dots,b_n)=0$ for all $b_i\in M_n(\F)$, we call $Q_n$ the {\em Cayley-Hamilton identity}. 
 Recall that $Q_n$ can be written as
\begin{equation}
\label{ch}Q_n=\sum_{\sigma\in S_{n+1}}(-1)^\sigma\phi_\sigma(x_1,\ldots,x_n)
\end{equation} where $(-1)^\sigma=\pm 1$ denotes the sign of the permutation $\sigma$ and  $\phi_\sigma$ is defined using the cycle decomposition of $$\sigma=(i_1,\ldots, i_{k_1})(j_1,\ldots ,j_{k_2})\ldots (u_1,\ldots ,u_{k_t})(s_1,\ldots s_{k }, n+1)$$ as
$$\phi_\sigma(x_1,\ldots,x_n) =\tr(x_{i_1} \cdots x_{i_{k_1}})\tr(x_{j_1} \cdots x_{j_{k_2}})\cdots \tr(x_{u_1} \cdots x_{u_{k_t}})x_{s_1}\cdots x_{s_{k }}.$$
Regarding ${\rm tr}(y)$ as $\sum_{i,j} e_{ij}y e_{ji}$, and consequently  ${\rm tr}(y_1\dots y_k)$ as $\sum_{i,j} e_{ij}y_1\dots y_k e_{ji}$, we see that we may consider  $Q_n$ as a GPI of $M_n(\F)$.
\begin{corollary}
 If {\rm char}$(\F) =0$,  then the T-ideal of all GPI's of $M_n(\F)$ is generated by the Cayley-Hamilton identity $Q_n$. 
\end{corollary}

\begin{proof}
It suffices to show that the basic identity 
$$[ex_1e,ex_2e]=ex_1ex_2e-ex_2ex_1e\in M_n(\F)\ast \F\X,$$
where $e$ is a rank one idempotent, follows from the Cayley-Hamilton identity. 
To this end we insert $[ex_1e,ex_2e]$ for   one variable   and $e$ for the others in $Q_n$. 
Note that $1$ needs to be in the last cycle of $\sigma$ for $\phi_\sigma([ex_1e,ex_2e],e,\dots,e)$ to be nonzero. 
In this case $\phi_\sigma([ex_1e,ex_2e],e,\dots,e)=[ex_1e,ex_2e]$. 
Thus, we need to count the number of such permutations with the corresponding signs. 

Take a cycle $\tau$. 
Note that permutations with the corresponding signs having the last cycle $\tau$ do not sum to zero only if $\tau$ is of length $n$ or $n+1$. 
For $\tau$ of length $n$ we have $(n-1)(n-1)!$ permutations of sign $(-1)^{n-1}$, while for $\tau$ of length $n+1$ the number of permutation is $n!$ and all have sign $(-1)^{n}$. Hence, 
$$Q_n\big([ex_1e,ex_2e],e,\dots,e\big)
=(-1)^{n}(n-1)![ex_1e,ex_2e].$$
\end{proof}

\section{One-sided functional identities and syzygies}\label{stri}
One-sided functional identities are intimately connected with syzygies on generic matrices. 
We first recall a result on syzygies that will be used in the sequel. 
\subsection{Syzygies on a generic matrix}
First we
introduce some auxiliary notation. Let $r,s\ge 1$, let  $C_y=\F[y_{ij}\mid 1\leq i\leq \n,1\leq j\leq \m]$ be a polynomial algebra, 
and let $u_1,\dots,u_\n$ be generators of the free module $C_y^n$. 
The matrix of independent commutative variables $Y=(y_{ij})_{1\leq i\leq \n,1\leq j\leq \m}$ is called a generic $\n\times \m$-matrix. 
With a syzygy on $Y$ we mean a syzygy on the rows of $Y$; i.e., an element $\sum_{i=1}^\n f_i u_i\in C_y^\n$ with the property $\sum_{i=1}^\n f_iy_{ij}=0$ for $j=1,\dots, \m$. If $\m\leq \n$ we denote by $[ j_1,\dots,j_\m]$, 
where $1\leq j_1<\cdots<j_\m\leq \n$,
  the determinant of the $\m\times \m$-submatrix of $Y$, obtained by restricting $Y$ to the rows indexed by $j_1,\dots,j_\m$. 

\begin{theorem}[{\cite[Theorem 7.2]{On}}]\label{onn}
Let $1\leq \m< \n$ and let $Y$ be a generic $\m\times \n$-matrix. 
The set of determinantal relations
$$
G=\bigg\{\sum_{\ell=0}^\m(-1)^\ell[ j_0,\dots,\hat j_\ell,\dots,j_\m]u_{j_\ell}\Bigm\vert 1\leq j_0<j_1<\cdots<j_\m\leq n\bigg\}
$$
generates the module of syzygies on $Y$ and is a Gr\"{o}bner basis for it with respect to any lexicographic monomial order on $C_y^\n$.
\end{theorem}

\subsection{Solving left-sided functional identities}
We restrict ourselves to the {\em left-sided functional identities}; i.e., functional identities of the form $\sum_{k\in K} F_k(\overline{x}_m^k)x_k=0$. The right-sided functional identities  $\sum_{l\in L}x_lG_l(\overline{x}_m^l)=0$ can be of course treated in the same way. Besides, by applying the transpose operation to a  right-sided functional identity we obtain a left-sided functional identity, so that the results that we will obtain are more or less directly applicable to  right-sided functional identities.


The standard solution of $\sum_{k\in K} F_k(\overline{x}_m^k)x_k=0$ is defined as $F_k=0$ for each $k\in K$.  If $|K|\le n$, then there are no other solutions on $M_n(\F)$ -- this is an easy consequence of the general theory of functional identities (see e.g.~ \cite[Corollary 2.23]{FIbook}). 
To obtain a nonstandard solution we have to modify the Cayley-Hamilton identity. We take only its noncentral part and commute it with the product of a fixed matrix and a new variable. 
In this way we arrive at a basic functional identity on $M_n(\F)$ in  $n+1$ variables 
\beq \label{n-std}
\big[\td Q_n(a_1x_1,\dots,a_nx_n),a_{n+1}x_{n+1}\big]=0, 
\eeq
where $\td Q_n$ denotes the noncentral part of $Q_n$; i.e.,
$$\td Q_n(x_1,\dots,x_n)=\sum_{\sigma\in S_{n+1}\setminus S_n}\epsilon_\sigma\phi_\sigma(x_1,\ldots,x_n),$$
and $a_i\in M_n(\F)$. 
Note that \eqref{n-std} can be indeed interpreted as  a left-sided functional identity.
For example, in the case $n=2$ we have 
$$\td Q_2(x_1,x_2)=x_1x_2+x_2x_1-\tr(x_1)x_2-\tr(x_2)x_1,$$
so that
\begin{multline*}
\big[\td Q_2(a_1x_1,a_2x_2),a_3x_3\big]= \Big(
a_3x_3(-a_2x_2+\tr(a_2x_2))a_1\Big)x_1+ \Big(a_3x_3(-a_1x_1+\tr(a_1x_1))a_2\Big)x_2\\+ \Big(\bigl(a_1x_1a_2x_2+a_2x_2a_1x_1-\tr(a_1x_1)a_2x_2-\tr(a_2x_2)a_1x_1\bigr)a_3\Big)x_3 =0.
\end{multline*}
If we multiply \eqref{n-std} by a scalar-valued function $\lambda(x_{n+2},\dots,x_m)$, we get a left-sided functional identity of $M_n(\F)$ in $m$ variables for an arbitrary $m > n$.
Of course, by permuting the variables $x_i$ we get further examples. 
 We will see that in fact every left-sided functional identity of $M_n(\F)$ is a sum of left-sided functional identities of such a type.

Before giving a full description of the unknown functions $F_k$ satisfying $\sum_{k\in K} F_k(\overline{x}_m^k)x_k=0$ we need to introduce some more notation. 
Let $$C=\F\big[x_{ij}^{(k)}\mid 1\leq i,j\leq n, 1\leq k\leq \g\big].$$ 
We denote  by 
$$
[(i_1,j_1),\dots,(i_n,j_n)]
$$
the determinant of the matrix  with the $\ell$-th row equal to the $j_\ell$-th row of the generic matrix $X_{i_\ell}=(x_{ij}^{(i_\ell)})$. 
For example, in the case $n=2$, $[(1,1),(2,1)]$ denotes the determinant of the matrix 
$$
\left(\begin{array}{cc}
x_{11}^{(1)}&x_{12}^{(1)}\\
x_{11}^{(2)}&x_{12}^{(2)}
\end{array}\right),
$$ 
and thus equals $x_{11}^{(1)}x_{12}^{(2)}-x_{12}^{(1)}x_{11}^{(2)}$. Further, $\N_d$ denotes the set $\{1,\dots,d\}$. If $I = \{i_1,\dots,i_p\}\subseteq\N_m$ with $i_1<\dots<i_{p}$, then $\ol x^I_m$
stands for $(x_1,\dots,\hat x_{i_1},\dots,\hat x_{i_p},\dots, x_m)$. Finally, as above by $e_{ij}$ we denote the matrix units.

Without loss of generality we may assume that $K=\N_m$. When dealing with two-sided functional identities \eqref{e1} it is often convenient to deal with the case where $K$ and $L$ are arbitrary subsets
of $\N_m$, but for one-sided identities this would only cause notational complications.

\begin{proposition}\label{1sided}
Let $\sum_{k=1}^mF_k(\overline{x}_m^k)x_k=0$ be a functional identity of $M_n(\F)$, where $F_k: M_n(\F)^{\g-1}\to M_n(\F)$ are multilinear functions.
Then there exist  multilinear scalar-valued functions $\lambda_{\ell IJ}$ such that
$$F_k(\overline{x}_m^k)=\sum_{\ell,I, J}(-1)^s\lambda_{\ell I J}(\ol x^I_m)\big[ (i_1,j_1),\dots,\widehat{(i_{s},j_{s})},\dots,(i_{n+1},j_{n+1})\big]e_{\ell j_s},$$
where the sum runs over all $\ell\in\N_n$, all $I=\{i_1,\dots,i_{n+1}\}\subset \N_\g$ such that $i_s=k$ for some $s$, $i_1<\dots<i_{n+1}$, and all $J=(j_1,\dots,j_{n+1})\in \N_n^{n+1}$.
\end{proposition}

\begin{proof}
We can view  $\sum_{k=1}^mF_k(\overline{x}_m^k)x_k=0$ as the identity in $M_n(C)$, 
since $F_k$ are  multilinear and hence polynomial maps.
We then have an identity of the form
$$\sum_{k=1}^\g H_{k}X_k=0,$$
where $X_k=(x_{ij}^{(k)})$, $1\leq k\leq \g$, are generic matrices, and $H_{k}\in M_n(C)$ correspond to $F_k$. 
Note that this identity is equivalent to $n$ identities
$$\sum_{k=1}^\g e_{\ell\ell}H_{k}X_k=0, \quad 1\leq \ell\leq n.$$
We thus first find solutions of each of them. 
Without loss of generality we assume that $\ell=1$ and $H_k=e_{11}H_k$, $1\leq k\leq \g$. 
We can further rewrite this identity as 
$$\sum_{k=1}^\g\sum_{j=1}^n H^{(k)}_{1j}x_{j r}^{(k)}=0$$
for  $1\leq r \leq n$, which can be viewed as 
$$
\Big(H^{(1)}_{11},\dots,H^{(1)}_{1n},\;\dots\;,H^{(\g)}_{11},\dots,H^{(\g)}_{1n}\Big)
\left(
\begin{array}{ccc}
x^{(1)}_{11}&\dots&x^{(1)}_{1n}\\
\vdots&\ddots&\vdots\\
x^{(1)}_{n1}&\dots&x^{(1)}_{nn}\\
&&\\
&\vdots&\\
&&\\
x^{(\g)}_{11}&\dots&x^{(\g)}_{1n}\\
\vdots&\ddots&\vdots\\
x^{(\g)}_{n1}&\dots&x^{(\g)}_{nn}\\
\end{array}
\right)=(0\dots 0).
$$
We write $X_{n,\g}$ for the matrix on the right. 
By definition, $(H^{(1)}_{11},\dots,H^{(1)}_{1n},\;\dots\;,H^{(\g)}_{11},\dots,H^{(\g)}_{1n})$ is a syzygy on the rows of the generic matrix $X_{n,\g}$. 
The Gr\"obner basis of the module of syzygies is described in Theorem \ref{onn}. 
Since in our case $H^{(k)}_{ij}$ are multilinear as functions of $X_1,\dots,X_{k-1},X_{k+1},\dots,X_\g$, 
we look for the elements in the Gr\"obner basis with the same property. 
Note that 
$$
\big[ (i_1,j_1),\dots,(i_n,j_n)\big]
$$
denotes the determinant of the submatrix of $X_{n,\g}$ with the $k$-th row equal to $j_k$-th row of the generic matrix $X_{i_k}$.
By  Theorem \ref{onn} the desired generators are 
$$\sum_{k=1}^{n+1}(-1)^k\big[(i_1,j_1),\dots,\widehat{(i_{k},j_{k})},\dots,(i_{n+1},j_{n+1})\big]u_{(i_k-1)n+j_k}$$
for $1\leq i_1<i_2<\cdots<i_{n+1}\leq \g$,  $1\leq j_1,\dots,j_{n+1}\leq  n$, and the proposition follows.
\end{proof}

\subsection{Left-sided functional identity as a GPI}\label{subs3}
Each function $F_k$ is multilinear so it corresponds to a multilinear generalized polynomial $f_k\in M_n(\F)\ast \F\X$ such that the evaluation of $f_k$ 
on $M_n(\F)$ coincides with $F_k$. Note that this correspondence is uniquely determined  up to the generalized polynomial identities. At any rate, the
problem of describing functional identities $\sum_{k\in K} F_k(\overline{x}_m^k)x_k=0$ is basically equivalent to the problem of describing GPI's of the form $\sum_{k\in K} f_k(\overline{x}_m^k) x_k$. In this section we will deal with the latter since the GPI setting seems to be more convenient for formulating the main result. 

We start with a technical lemma. 
By $D_{j_1,\dots,j_n}$ we denote the determinant of the matrix whose $k$-th row is equal to the $j_k$-th row of the generic matrix $X_k$.    

\begin{lemma}\label{detCH}
Let $i_\ell,j_\ell\in \N_n$, $1\leq \ell\leq n+1$. If $\{i_1,\dots ,i_n\}=\{1,\dots,n\}$ then
\begin{align*}\label{forchdet}
&e_{i_{n+1},j_{n}}x_{n}Q_{n-1}(e_{i_1 j_1}x_1,\dots,e_{i_{n-1},j_{n-1}}x_{n-1})e_{i_nj_{n+1}}\\
=&(-1)^\tau D_{j_1,\dots,j_{n}}e_{i_{n+1},j_{n+1}}\\
=&-e_{i_{n+1},j_{n+1}}\td Q_n(e_{i_1j_1}x_1,\dots,e_{i_nj_n}x_n),
\end{align*}
where $\tau\in S_n$ is given by $\tau(k)=i_k$, otherwise $Q_{n-1}(e_{i_1,j_1}x_1,\dots,e_{i_{n-1},j_{n-1}}x_{n-1})e_{i_n,j_n}=0$.
\end{lemma}

\begin{proof}
The last assertion follows from the fact that $Q_{n-1}$ is an identity of $M_{n-1}(\F)$. 
Indeed, we may assume without loss of generality that $i_k=k$ for $1\leq k\leq s<n$, $i_s=\dots=i_n=s$, and 
hence
$$Q_{n-1}(e_{1j_1}x_1,\dots,e_{sj_s}x_s,\dots,e_{sj_{n-1}}x_{n-1})e_{sj_n}\in \big(M_n(\F)e_{nn}\big)e_{sj_n}=\{0\}.$$ 
The first equality  clearly follows by using
 the identity 
$$e_{i_{n+1},j_n}x_n\phi_\sigma(e_{i_1j_1}x_1,\dots,e_{i_{n-1},j_{n-1}}x_{n-1})e_{i_n,j_{n+1}}=x_{j_1i_{\s(1)}}^{(1)}\cdots x_{j_ni_{\s(n)}}^{(n)}e_{i_{n+1},j_{n+1}}$$
for $\s\in S_n$ in the expression \eqref{ch} of $Q_{n-1}$, and the second one follows from 
$$(-1)^\tau D_{j_1,\dots,j_{n}}=-\td Q_n(e_{i_1j_1}x_1,\dots,e_{i_nj_n}x_n),$$ 
 which can be deduced in a similar way after applying the identity
$$\td Q_n(y_1,\dots,y_n)=-\sum_{\s\in S_n\subset S_{n+1}}(-1)^\s \phi_\s(y_1,\dots,y_n).$$
\end{proof}



Let us call $g\in M_n(\F)\ast \F\X$  a {\em central generalized polynomial} if all its evaluations on 
$M_n(\F)$ are scalar matrices. For instance, $\td Q_n$ is a central generalized polynomial.

\begin{theorem}\label{FIGPI}
Let $f_1,\dots,f_m\in  M_n(\F)\ast \F\X$  be multilinear generalized polynomials such that $P=\sum_{k=1}^m f_k(\overline{x}_m^k) x_k$ is a GPI of $M_n(\F)$. Then $P$ can be written as a sum of GPI's of the form  
$$g  \big[\td Q_n(a_{1}x_{k_1},\dots,a_{n}x_{k_n}),a_{n+1}x_{k_{n+1}}\big],
$$
where  $k_i\ne k_j$ if $i\ne j$, $a_{i}\in M_n(\F)$, and $g$ is a multilinear central generalized polynomial (in all variables except $x_{k_i}$).
\end{theorem}

\begin{proof}
Proposition \ref{1sided} implies that $P$ can be written as a sum of GPI's of the form 
$$\sum_{k=1}^{n+1}(-1)^k \big[(i_1,j_1),\dots,\widehat{(i_k,j_k)},\dots,(i_{n+1},j_{n+1})\big]e_{\ell j_k}x_{i_k}, $$
multiplied by central generalized polynomials; the determinants appearing in the identity can also be viewed as central generalized polynomials. It is thus enough to prove that this identity can be written in the desired way. 
We can assume without loss of generality that $\ell=1$, $i_k=k$, $1\leq k\leq n+1$. 
Hence we can write the identity as 
$$
\sum_{k=1}^{n+1}(-1)^kD_{j_1,\dots,\hat j_k,\dots,j_{n+1}}e_{1j_k}x_k,
$$
where $D_{j_1,\dots,\hat j_k,\dots, j_{n+1}}$ stands for the determinant of the submatrix of the matrix that has the $\ell$-th row  equal to the $j_\ell$-th row of the generic matrix $X_\ell$, in which we remove the $k$-th row.
Note that 
$$\td Q_n(x_1,\dots,x_n)=\sum_{k=1}^n Q_{n-1}(x_1,\dots,\hat x_k,\dots,x_n)x_k.$$
Applying Lemma \ref{detCH} we thus obtain 
\begin{align*}
e_{1j_{n+1}}x_{n+1}\td  Q_n(e_{1j_1}x_1,&\dots,e_{nj_n}x_n)\\
=&e_{1j_{n+1}}x_{n+1} \sum_{k=1}^n Q_{n-1}(e_{1j_1}x_1,\dots,\widehat {e_{kj_k}x_k},\dots,e_ {nj_n}x_n)e_{kj_k}x_k\\
=&\sum_{k=1}^n (-1)^{\tau_k} D_{j_1,\dots,\hat j_k,\dots,j_{n+1}} e_{1j_k} x_k,
\end{align*}
where $\tau_k=(k,k+1,\dots,n)$, and thus $(-1)^{\tau_k}=(-1)^{n-k}$. 
Applying Lemma \ref{detCH} we obtain 
\begin{multline*}
\sum_{k=1}^{n+1}(-1)^k D_{j_1,\dots,\hat j_k,\dots,j_{n+1}} e_{1j_k} x_k\\
= (-1)^n e_{1j_{n+1}}x_{n+1}\td Q_n(e_{1j_1}x_1,\dots,e_{nj_n}x_n)-(-1)^nQ_n(e_{1j_1}x_1,\dots,e_{nj_n}x_n)e_{1j_{n+1}}x_{n+1}, 
\end{multline*}
which yields the desired conclusion.
\end{proof}

\begin{remark}
Let us remark that the identity \eqref{n-std} can be written as 
$$\sum_{k=1}^{n+1}\td Q_n(a_1x_1,\dots,\widehat{a_kx_k},\dots,a_nx_n,\td a_{n+1,k}x_{n+1})\td a_kx_k=0$$
for some $\td a_k,\td a_{n+1,k} \in M_n(\F)$.
It is enough to establish the statement in the case when $a_k=e_{i_k,j_k}$, $1\leq k\leq n+1$, are matrix units. 
Applying Lemma \ref{detCH} we obtain, similarly as in the proof of Corollary \ref{FIGPI}, the identity
\begin{align*}
e_{i_{n+1}j_{n+1}}x_{n+1}\td  Q_n(e_{i_1j_1}x_1,&\dots,e_{i_nj_n}x_n)\\
=&e_{i_{n+1}j_{n+1}}x_{n+1} \sum_k Q_{n-1}(e_{i_1j_1}x_1,\dots,\widehat {e_{i_kj_k}x_k},\dots,e_ {i_nj_n}x_n)e_{i_kj_k}x_k\\
=&\sum_k (-1)^\tau D_{j_1,\dots,\hat j_k,\dots,j_{n+1}} e_{i_{n+1},j_k} x_k\\
=&\sum_k (-1)^{n}\td Q_n(e_{i_1j_1}x_1,\dots,\widehat{e_{i_k,j_k}x_k},\dots,e_{i_{k}j_{n+1}}x_{n+1})e_{i_{n+1}j_k}x_k,
\end{align*}
where $\tau\in S_n$, $\tau(k)=i_k$.
\end{remark}

\subsection{An application: Characterization of the determinant}\label{subdet}

The determinants have appeared throughout our discussion on one-sided identities. To point out their role more clearly, we record  two simple corollaries at the end of the section. 

 Let $A$ be an algebra over a field $\F$. 
A function $T:A\to A$ is said to be the {\em trace  of an $r$-linear function} $F:A^r\to A$ if $T(x) = F(x,\dots,x)$ for all $x\in A$ (if $r=0$ then $T$ is a constant function). 
We remark that if $F$ is symmetric and char$(F)$ is $0$ or greater than $r$, then $F$ is uniquely determined by its trace $T$. This can be shown by applying the linearization process.    

\begin{corollary}\label{1sidedc}
Let $m > n$ and let  $T_k:M_n(\F)\to M_n(\F)$ be the trace of an $(m-1)$-linear function
  $F_k: M_n(\F)^{\g-1}\to M_n(\F)$, $1\le k\le m$.
If $\sum_{k=1}^mF_k(\overline{x}_m^k)x_k=0$ is a functional identity of $M_n(\F)$, then each $T_k$ can be written as
$T_k(x) = \det(x)S_k(x)$ where $S_k$ is the trace of an $(m-1-n)$-linear function.
\end{corollary}

\begin{proof}
 Proposition \ref{1sided} shows  that  $T_k(x_1)$
is a sum of  the  terms of the form 
$$(-1)^s\lambda_{\ell I J}(x_1,\dots,x_1) \big[ (1,j_1),\dots,\widehat{(1,j_{s})},\dots,(1,j_{n+1})\big]e_{\ell j_s}.$$
By definition, $[ (1,j_1),\dots,\widehat{(1,j_{s})},\dots,(1,j_{n+1})\big]=\pm \det(x_1)$ if $j_1,\dots,\hat j_s,\dots,j_{n+1}\in \N_n$ are  distinct, and $0$ otherwise. This proves the corollary.
\end{proof}

In the final corollary we add more assumptions  in order to obtain an abstract  characterization of the determinant.

\begin{corollary}\label{1sidedc2}
Let $F:M_n(\F)^{n} \to M_n(\F)$ be an $n$-linear function satisfying $F(1,\dots,1) = 1$, and let $T$ be the trace of $F$. If {\rm char}$(\F)= 0$ or {\rm char}$(\F)> n$, then the following statements are equivalent:
\begin{enumerate}
\item[{\rm (i)}] There exist multilinear functions $F_1,\dots,F_n:M_n(\F)^{n} \to M_n(\F)$ such that 
\begin{equation}\label{fun}
\sum_{k=1}^{n}F_k(\overline{x}_{n+1}^k)x_k + F(\overline{x}_{n+1}^{n+1})x_{n+1}=0
\end{equation}
is a functional identity of $M_n(\F)$.
\item[{\rm (ii)}]  $T(x) = \det(x)\cdot 1$.
\end{enumerate}
\end{corollary}

\begin{proof}
Corollary \ref{1sidedc} shows that (i) implies (ii). To establish the converse, just note that $\det(x)\cdot 1$ is the trace of 
the function $\frac{1}{n!}\tilde{Q}_n(x_1,\dots,x_n)$, and that $\frac{1}{n!}[\tilde{Q}_n(x_1,\dots,x_n),x_{n+1}]=0$ can be written in the form \eqref{fun} with $F(\overline{x}_{n+1}^{n+1})= \frac{1}{n!}\tilde{Q}_n(\overline{x}_{n+1}^{n+1})$.
\end{proof}

\section{Two-sided functional identities}\label{s4}

In this section we consider the general  two-sided functional identities \begin{equation}\label{e11}
\sum_{k\in K}F_k(\overline{x}_m^k)x_k = \sum_{l\in L}x_lG_l(\overline{x}_m^l).
\end{equation}
 Let us first examine what result can be expected.

\subsection{Solutions of two-sided  functional identities}\label{sub41}

Let us first consider \eqref{e11} in an arbitrary algebra $A$ with center $Z$.
Suppose there exist multilinear functions
\begin{align*}
&p_{kl}:A^{m-2}\to A, \,\,k\in K,\,\,l\in
L,\,\,k\not=l,\\
&\lambda_i:A^{m-1}\to Z,\,\,i\in K\cup L,
\end{align*}
such that
\begin{align}\label{es}
F_k(\ol{x}_m^k)  &=  \sum_{l\in L, \,l\not=k}x_lp_{kl}(\ol{x}_m^{kl})
+\lambda_k(\ol{x}_m^k),\quad k\in K,\nonumber\\
G_l(\ol{x}_m^l) & =  \sum_{k\in K, \,k\not=l}p_{kl}(\ol{x}_m^{kl})x_k +
\lambda_l(\ol{x}_m^l),\quad l\in L,\\
\lambda_i&=0\quad\mbox{if}\quad i\not\in K\cap L.\nonumber
\end{align}
Note that then \eqref{e11} is fulfilled. We call \eqref{es} a {\em standard solution} of \eqref{e11}. In a large class of algebras a standard solution is also the only possible solution of \eqref{e11} \cite{FIbook}. Especially in infinite dimensional algebras this often turns out to be the case. For $A=M_n(\F)$, this holds provided only if $|K|\le n$ and $|L|\le n$ (see e.g.~ \cite[Corollary 2.23]{FIbook}).
Indeed,  if $|K| > n$ or $|L| >n$ then there exist nonstandard solutions of one-sided identities. Do all nonstandard solutions of \eqref{e11} in $M_n(\F)$ arise from the one-sided identities?
More specifically, let us call a solution of \eqref{e11} a {\em standard solution modulo one-sided identities} if there exist multilinear functions 
\begin{align*}
&\varphi_k,\psi_l:A^{m-1}\to A, \,\,k\in K,\,\, l\in L,\\
&p_{kl}:A^{m-2}\to A, \,\,k\in K,\,\,l\in
L,\,\,k\not=l,\\
&\lambda_i:A^{m-1}\to Z,\,\,i\in K\cup L,
\end{align*}
such that
\begin{eqnarray}\label{es1}
&F_k(\ol{x}_m^k)  =   \displaystyle \sum_{l\in L,\,l\not=k}x_lp_{kl}(\ol{x}_m^{kl}) 
+\lambda_k(\ol{x}_m^k) + \varphi_k(\ol{x}_m^k),\quad k\in K,\nonumber\\
&
G_l(\ol{x}_m^l) = \displaystyle \sum_{k\in K,\,
k\not=l}p_{kl}(\ol{x}_m^{kl})x_k +
\lambda_l(\ol{x}_m^l) + \psi_l(\ol{x}_m^l),\quad l\in L,\\
&   \lambda_i=0\quad\mbox{if}\quad i\not\in K\cap L,\nonumber\\
   &\displaystyle\sum_{k\in K}\varphi_k(\overline{x}_m^k)x_k = 
  \displaystyle\sum_{l\in L}x_l\psi_l(\overline{x}_m^l) =0.\nonumber
\end{eqnarray}
This notion is not vacuous. Namely, there do exist algebras admitting functional identities \eqref{e11} having solutions that are not standard modulo one-sided identities. One can actually find algebras with this property in which all solutions of one-sided identities are standard; see e.g.~ \cite[Example 5.29]{FIbook}.
 
Our goal in the rest of the paper is to show that every solution of \eqref{e11} on $A=M_n(\F)$ is standard  modulo one-sided identities.  As solutions of one-sided identities have been described in the preceding section, this will give a complete solution of the problem  to which we have addressed ourselves in this paper.

\subsection{Gr\"obner basis of a module representing functional identities}
In this  subsection we put functional identities aside, 
and establish auxiliary results needed for the proof of the main result, Theorem \ref{2fi}.

First we introduce the necessary framework.
We take $m,n\in \N$ and define $n^2\times n^2$-matrices 
$$
X'_k=\left(\begin{array}{ccccccc}
x_{11}^{(k)} &\dots&x_{1n}^{(k)}&&&&\\
&\vdots&&&&\\
x_{n1}^{(k)} &\dots&x_{nn}^{(k)}&&&&\\
&&&\ddots&&&\\
&&&&x_{11}^{(k)} &\dots&x_{1n}^{(k)}\\
&&&&&\vdots&\\
&&&&x_{n1}^{(k)} &\dots&x_{nn}^{(k)}\\
\end{array}\right)
=X_k\tnz 1\in M_n(\F)\tnz M_n(\F),
$$
$$
X''_k=\left(\begin{array}{ccccccc}
x_{11}^{(k)} &&&&x_{n1}^{(k)}&&\\
&\ddots&&\cdots&&\ddots&\\
&&x_{11}^{(k)} &&&&x_{n1}^{(k)}\\
&&&\vdots&&&\\
x_{1n}^{(k)} &&&&x_{nn}^{(k)}&&\\
&\ddots&&\cdots&&\ddots&\\
&&x_{1n}^{(k)} &&&&x_{nn}^{(k)}\\
\end{array}\right)
=1\tnz X_k^t\in M_n(\F)\tnz M_n(\F),
$$
and write 
$$\Xi=(X'_1,\dots,X'_m,X''_1,\dots,X''_m)^t.$$
Let $M$ be the submodule of the  module $C^{n^2}$ generated by the rows of $\Xi$. 
We aim to find a Gr\"obner basis of $M$. 

We need some more  notation. 
Let 
$K=\{k_1,\dots,k_a\}$ and $L=\{l_1,\dots,l_b\}$ be subsets of $\N_m$, and let us write $Q=K\cap L=\{q_1,\dots,q_c\}$ (this set may be empty).
 Let ${(d_1,f_1),\dots,(d_c,f_c)}$ be such that 
$q_{\ell}=k_{d_\ell}=l_{f_\ell}$. 
We attach the tuples   $V=(v_1,\dots,v_a)\in \N_n^a$ and $S=(s_1,\dots,s_b)\in \N_n^b$ to $K$ and $L$, resp.
For a subset $U\subset \N_a$, $|U|=a-c$, we write  
$U^\mathsf{c}=\{u'_1,\dots,u'_c\}$, $u'_1<\cdots<u'_c$, for the complement of $U$ in $\N_a$. Let $\s$ belong to $\sym U^\com$, the permutation group of $U^\com$.
We choose $\lambda\in \N_n$ and write 
$$D\cc_\lambda(Q_\s,L_S\setminus Q)$$
for the determinant of the $b\times b$-matrix 
$Y=(y_{ij})$, where 
\begin{align*}
y_{i\ell}&=\left\{
\begin{array}{ll}
x_{is_\ell}^{(l_\ell)} &\textrm{for $1\leq i\leq b-1$, $\ell\in \N_b\setminus \{f_1,\dots,f_c\}$, }\\
x_{\lambda s_\ell}^{(l_\ell)}&\textrm{for $i=b$, $\ell\in \N_b\setminus \{f_1,\dots,f_c\}$,}
\end{array}
\right.\\
y_{if_\ell}&=\left\{
\begin{array}{ll}
x_{i\s(u'_\ell)}^{(q_\ell)} &\textrm{for $1\leq i\leq b-1$, $1\leq\ell\leq c$, }\\
x_{\lambda\s(u'_\ell)}^{(q_\ell)}&\textrm{for $i=b$,  $1\leq\ell\leq c$.}
\end{array}
\right.
\end{align*}
Analogously, let $\tau\in \sym W^\com$, where $W^\com=\{w'_1,\dots,w'_c\}\subseteq \N_b$, and write 
$$D\rr_\lambda(Q_\tau,K_V\setminus Q)$$
for the determinant of the $a\times a$-matrix 
$Z=(z_{ij})$, where 
\begin{align*}
z_{\ell j}&=\left\{
\begin{array}{ll}
x_{v_\ell j}^{(k_\ell)} &\textrm{for $1\leq j\leq a-1$, $\ell\in \N_a\setminus \{d_1,\dots,d_c\}$, }\\
x_{v_\ell \lambda}^{(k_\ell)}&\textrm{for $j=a$, $\ell\in \N_a\setminus \{d_1,\dots,d_c\}$,}
\end{array}
\right.\\
z_{d_\ell j}&=\left\{
\begin{array}{ll}
x_{\tau(w'_\ell)j}^{(q_\ell)} &\textrm{for $1\leq j\leq a-1$, $1\leq\ell\leq c$, }\\
x_{\tau(w'_\ell)\lambda}^{(q_\ell)}&\textrm{for $j=a$, $1\leq\ell\leq c$.}
\end{array}
\right.
\end{align*}
In particular, we write $D^{\bf c}_{\lambda}(L_S)$, $D^{\bf r}_{\lambda}(K_V)$ for $D\cc_\lambda(\emptyset,L_S)$,  $D\rr_\lambda(\emptyset,K_V)$, resp.
(Note that $Y$ and $Z$ are formed from the columns  (resp. rows) of certain matrices, which is the reason for using $\cc$ (resp. $\rr$) in the above notation.)

We further denote by
$$d\cc_{\lambda, W^\com}(Q_\s),\quad d\cc_{\lambda,W}(L_S\setminus Q)$$
  the determinant of the submatrix of  $Y$ containing the  columns labeled by $f_1,\dots,f_\ell$ (resp. by $\ell\in\N_b\setminus \{f_1,\dots,f_\ell\}$) and rows labeled by $i_\ell\in W^\com$ (resp. $i_\ell\in W$), 
and by
$$d\rr_{\lambda,U^\com}(Q_\tau),\quad d\rr_{\lambda,U}(K_V\setminus Q)$$
  the determinant of the submatrix of $Z$ containing  the  rows labeled by $d_1,\dots,d_\ell$ (resp. by $\ell\in \N_a\setminus\{d_1,\dots,d_\ell\}$) and columns labeled by $j_\ell\in U^\com$ (resp. $j_\ell\in U$). 
We let the determinant of the empty matrix be $1$. 

\begin{example}
Let $n=4$, $K=\{1,2,3\}$, $L=\{2,3,4,5\}$, $V=(4,1,2)$, $S=(3,4,2,1)$, $U=\{2\}$, $W=\{1,3\}$, $\s=\id$, $\tau=(24)$, $\lambda=4$.
Then $Q=\{2,3\}$,  
$$
D\cc_\lambda(Q_\s,L_S\setminus Q)=
\left|\begin{array}{cccc}
x_{11}^{(2)}&x_{13}^{(3)}&x_{12}^{(4)}&x_{11}^{(5)}\\
x_{21}^{(2)}&x_{23}^{(3)}&x_{22}^{(4)}&x_{21}^{(5)}\\
x_{31}^{(2)}&x_{33}^{(3)}&x_{32}^{(4)}&x_{31}^{(5)}\\
x_{41}^{(2)}&x_{43}^{(3)}&x_{42}^{(4)}&x_{41}^{(5)}\\
\end{array}\right|,\quad
D\rr_\lambda(Q_\tau,K_V\setminus Q)=
\left|\begin{array}{ccc}
x_{41}^{(1)}&x_{42}^{(1)}&x_{44}^{(1)}\\
x_{41}^{(2)}&x_{42}^{(2)}&x_{44}^{(2)}\\
x_{21}^{(3)}&x_{22}^{(3)}&x_{24}^{(3)}
\end{array}\right|,
$$
$$
d\cc_{\lambda, W^\com}(Q_\s)=
\left|\begin{array}{cc}
x_{21}^{(2)}&x_{23}^{(3)}\\
x_{41}^{(2)}&x_{43}^{(3)}\\
\end{array}\right|,\quad
 d\cc_{\lambda,W}(L_S\setminus Q)=
\left|\begin{array}{cc}
x_{12}^{(4)}&x_{11}^{(5)}\\
x_{32}^{(4)}&x_{31}^{(5)}\\
\end{array}\right|,
$$
$$
d\rr_{\lambda,U^\com}(Q_\tau)=
\left|\begin{array}{cc}
x_{41}^{(2)}&x_{44}^{(2)}\\
x_{21}^{(3)}&x_{24}^{(3)}
\end{array}\right|,\quad
d\rr_{\lambda,U}(K_V\setminus Q)=
\left|\begin{array}{ccc}
x_{42}^{(1)}
\end{array}\right|.
$$
\end{example}
Let $u_{\gamma,\delta}$ denote the $n(\gamma-1)+\delta$-th basis element in the $C$-module $C^{n^2}$. 
We write 
\begin{align*}
G'&=\bigg\{\sum_{\alpha=|K|}^n D\rr_\alpha(K_V)u_{\gamma,\alpha}\mid \;\gamma\in \N_n,K\subset \N_m,V\subset \N_n, |V|=|K|\bigg\},\\
G''&=\bigg\{\sum_{\beta=|L|}^n D\cc_\beta(L_S)u_{\beta,\delta}\mid \;\delta\in \N_n,L\subset \N_m,S\subset \N_n, |S|=|L|\bigg\}.
\end{align*}

Note that every polynomial in $C$ can be treated as a function on $M_n(\F)^m$. 
Let $C_{mult}$ denote the elements in $C$ that are multilinear in some set of variables $x_{k_1},\dots,x_{k_\ell}$, $1\leq k_i\neq k_j\leq m$.  
We will say that $G$ is a {\em multilinear Gr\"obner basis}  of a $C$-module $N\subset C^r$ if there exists a Gr\"obner basis $\td G$ of $N$ such that  
$G=\td G \cap C_{mult}$. 
\begin{proposition}[{\cite[Theorem 8.4]{On}}]\label{multgb}
The set $G'$ is a multilinear Gr\"obner basis of the submodule $M'$ of $M$ generated by the first $mn^2$ rows of $\Xi$, and the set 
$G''$ is a multilinear Gr\"obner basis of the submodule $M''$ of $M$ generated by the last $mn^2$ rows of $\Xi$.
\end{proposition}

We will use this proposition in order to show that one can obtain a multilinear Gr\"obner basis of $M$ by simply joining $G'$ and $G''$. To establish this result in Lemma \ref{grob} we need  some preliminary lemmas. 

For a subset $U\subset \N_a$ we set $U_\alpha=U$ if $a\not\in U$, and $U_\alpha=(U\setminus\{a\})\cup \{\alpha\}$ if $a\in U$.

\begin{lemma}\label{zagrob}
Let $a,b,c\in \N$, $c\leq a, b\leq n$, $U\subset\N_a$, $W\subset \N_b$, $|U|=|W|=c$, $Q=\{q_1,\dots,q_c\}$. 
For $a\leq \alpha$, $b\leq \beta$ we have
\beq\label{Dvs}
\sum_{\s\in \sym U_\alpha}(-1)^\s d\cc_{\beta,W}(Q_\s)=\sum_{\tau\in \sym W_\beta}(-1)^\tau d\rr_{\alpha,U}(Q_\tau).
\eeq
\end{lemma}

\begin{proof}
Notice that we can assume without loss of generality that $U=W=Q=\{1,\dots,c\}$, $\alpha=\beta=c$. 
We compute 
\begin{align*}
\sum_{\s\in \sym U}(-1)^\s d\cc_{\beta,W}(Q_\s)&=\sum_{\s\in\sym c
}(-1)^\s\sum_{\rho\in\sym c}(-1)^\rho\, x_{1,\s\rho(1)}^{(\rho(1))}\cdots x_{c,\s\rho(c)}^{(\rho(c))}\\
&=\sum_{\rho\in\sym c}(-1)^{\rho^{-1}}\sum_{\s\in\sym c}(-1)^{\s^{-1}}\, x_{\rho^{-1}\s^{-1}(1),1}^{(\s^{-1}(1))}\cdots x_{\rho^{-1}\s^{-1}(c),c}^{(\s^{-1}(c))}\\
&=\sum_{\tau\in \sym W}(-1)^\tau d\rr_{\alpha,U}(Q_\tau).
\end{align*}
\end{proof}

\begin{lemma}\label{zagrob2}
Let $K=\{k_1,\dots,k_a\}$, $L=\{l_1,\dots,l_b\}$, $Q=K\cap L=\{q_1,\dots,q_c\}$. Let ${(d_1,f_1),\dots,(d_c,f_c)}$ be such that 
$q_{\ell}=k_{d_\ell}=l_{f_\ell}$, and let $V=(v_1,\dots,v_a)\in \N_n^a$, $S=(s_1,\dots,s_b)\in\N_n^b$.
For $a\leq \alpha\leq n$, $b\leq \beta \leq n$ we have   
\beq\label{grobid}
\begin{aligned}
&(-1)^{\sum d_\ell}\sum_{U\subset \N_a, |U|=a-c}(-1)^{\sum_{U^\com} u}\;d\rr_{\alpha,U}(K_V\setminus Q)\sum_{\s\in \sym (U^\com)_\alpha} (-1)^\s D\cc_\beta(Q_\s,L_S\setminus Q)\\
=&(-1)^{\sum f_\ell}\sum_{W\subset \N_b, |W|=b-c}(-1)^{\sum_{W^\com} w}\;d\cc_{\beta,W}(L_S\setminus Q)\sum_{\tau\in\sym (W^\com)_\beta} (-1)^\tau D\rr_\alpha(Q_\tau,K_V\setminus Q).
\end{aligned}
\eeq
\end{lemma}

\begin{proof}
Using the Laplace expansion by the columns $f_1,\dots,f_c$ we obtain
$$
D\cc_\beta(Q_\s,L_S\setminus Q)=\sum_{W\subset\N_b,|W|=b-c}(-1)^{\sum f_\ell}(-1)^{\sum_{W^\com} w} d\cc_{\beta,W}(L_S\setminus Q)d\cc_{\beta, W^\com}(Q_\s),
$$
and analogously using the Laplace expansion by the rows $d_1,\dots,d_c$ we have
$$
D\rr_\alpha(Q_\tau,K_V\setminus Q)=\sum_{U\subset\N_a,|U|=a-c}(-1)^{\sum d_\ell}(-1)^{\sum_{U^\com} u} d\rr_{\alpha,U}(K_V\setminus Q)d\rr_{\alpha,U^\com}(Q_\tau).
$$
By applying Lemma \ref{zagrob} we arrive at the desired conclusion. 
\end{proof}

Before proceeding to the proof of the next lemma we make a little digression and recall some facts concerning Gr\"obner bases and syzygies (see e.g.~\cite{Eis}). 
Let $A$ be a polynomial algebra. By $u_1,\dots,u_r$ we denote the generators of the free module $A^r$.  
Let $N$ be a module over $A$ and $\{g_1,\dots,g_t\}$ its Gr\"obner basis with respect to any monomial order on $A^r$.
Let $\inn(g_i)$ stand for the initial term of $g_i$. If $\inn(g_i)$ and $\inn(g_j)$ involve the same basis element of $A^r$, set
$$m_{ij}=\frac{\inn(g_i)}{{\rm GCD}(\inn(g_i),\inn(g_j))}\in A.$$ 
For each such pair $i,j$ choose an expression 
$$\s_{ij}:=m_{ji}g_i-m_{ij}g_j=\sum_\ell h_\ell^{(ij)}g_\ell,$$
 such that $\inn(\s_{ij})\geq \inn(h_\ell^{(ij)}g_\ell)$ for every $\ell$; it is
 called  a standard expression of $\s_{ij}$ in terms of the $g_\ell$, and its existence is guaranteed by the fact that $\{g_1,\dots,g_t\}$ is a Gr\"obner basis of $N$. 
For other pairs $i,j$ set $m_{ij}=0$, $h_\ell^{(ij)}=0$. 
By Shreyer's theorem (see e.g.~\cite[Theorem 15.10]{Eis}) the module of syzygies on the Gr\"obner basis $\{g_1,\dots,g_t\}$ of a module $N$ 
is generated by 
$$\tau_{ij}:=m_{ji}u_i-m_{ij}u_j-\sum_\ell h_\ell^{(ij)}u_\ell,$$
$1\leq i,j\leq t$. 

Let us define a monomial order $>$ on the module $C^{n^2}$. 
On $C$ we set $x_{i_1j_1}^{(k_1)}>x_{i_2j_2}^{(k_2)}$ if $(k_1,i_1,j_1)<(k_2,i_2,j_2)$ lexicographically (i.e. $k_1<k_2$, or $k_1=k_2$ and $i_1<i_2$, or $k_1=k_2$, $i_1=i_2$ and $j_1<j_2$). 
We define $pu_{\alpha,\beta}>qu_{\gamma,\delta}$ if $(\alpha,\beta,q)<(\gamma,\delta,p)$ lexicographically (i.e. $\alpha<\gamma$, or $\alpha=\gamma$ and $\beta<\delta$, or $\alpha=\gamma$, $\beta=\delta$ and $p>q$).

\begin{lemma}\label{grob}
The set $G'\cup G''$ is a multilinear Gr\"obner basis of $M$ with respect to the order $>$ on $C^{n^2}$.
\end{lemma}

\begin{proof}
Using the Buchberger's criterion (see e.g.~\cite[Theorem 15.8]{Eis}) together with Proposition \ref{multgb} we see that it suffices  to verify that $\sigma_{ij}$ has a standard expression in terms of $g_\ell\in G'\cup G''$ for $g_i\in G'$, $g_j\in G''$, initial terms of which involve the same basis element in $C^{n^2}$  
and  for which $\sigma_{ij}$ is multilinear. 
Choose $g_i\in G'$, $g_j\in G''$ such that $\inn(g_i)$ and $\inn(g_j)$ involve the same basis element of $C^{n^2}$. 
Define sets $K,L\subset \N_m$ such that  $g_i$ depends on the variables $x_{k_\ell}$ for $k_\ell\in K$, $g_j$ on $x_{l_\ell}$ for $l_\ell\in L$. 
Then the initial term involves the variables in $Q=K\cap L$.
For $\s_{ij}$ to be multilinear the factors in the initial terms of $g_i$ and $g_j$ dependent on the variables in $Q$ need to coincide. 
Hence, 
$$g_i=\sum_{\alpha=a}^n D\rr_\alpha(Q_\tau,K_V\setminus Q)u_{b,\alpha},$$
$$g_j=\sum_{\beta=b}^n D\cc_\beta(Q_\s,L_S\setminus Q)u_{\beta,a},$$
where $|K|=a$, $|L|=b$, $|Q|=c$, $V\in \N_n^a$, $S\in \N_n^b$, and 
 $W=\{d_1,\dots,d_c\}^\com$, $\tau=\id$,  $U=\{f_1,\dots,f_c\}^c$, $\s=\id$, 
and we have
$$\inn(g_i)=\prod_{\{\ell\mid\, k_{\ell}\in Q\}}x_{f_\ell,d_\ell}^{(q_\ell)}\prod_{\{\ell\mid\, k_\ell\in K\setminus Q\}}x_{v_\ell,\ell}^{(k_\ell)},\quad\quad
\inn(g_j)=\prod_{\{\ell\mid\, l_{\ell}\in Q\}}x_{f_\ell,d_\ell}^{(q_\ell)}\prod_{\{\ell\mid\, l_\ell\in L\setminus Q\}}x_{\ell,s_\ell}^{(l_\ell)}.$$
By Lemma \ref{zagrob2} we deduce
\begin{multline*}
\sum_{\lambda=b}^n  (-1)^{\sum f_\ell}\sum_{W\subset \N_b, |W|=b-c}(-1)^{\sum_{W^\com} w}\;d\cc_{\lambda,W}(L_S\setminus Q)\sum_{\tau\in\sym (W^\com)_\lambda} (-1)^\tau \sum_{\alpha=a}^n D\rr_\alpha(Q_\tau,K_V\setminus Q)u_{\lambda,\alpha}\\
=\sum_{\lambda=a}^n (-1)^{\sum d_\ell}\sum_{U\subset \N_a, |U|=a-c}(-1)^{\sum_{U^\com} u}\;d\rr_{\lambda,U}(K_V\setminus Q)\sum_{\s\in \sym (U^\com)_\lambda} (-1)^\s \sum_{\beta=b}^n D\cc_\beta(Q_\s,L_S\setminus Q)u_{\beta,\lambda}.
\end{multline*}
Indeed, restricting to the basis element $u_{\gamma,\delta}\in C^{n^2}$ on  both sides of this identity we get the identity \eqref{grobid} for $\alpha=\delta,\beta=\gamma$.
It remains to prove that this identity induces a standard expression for $\sigma_{ij}$. 
It is enough to check that the initial terms of the elements in the Gr\"obner basis (except for $g_i$ and $g_j$) that appear in this identity and involve the same basis element of $C^{n^2}$  are smaller or equal to $\inn(\s_{ij})$. 
Those are 
\begin{align}\label{inn1}
\inn\Big(d\cc_{\lambda,W}(L_S\setminus Q)D\rr_{a}(Q_\tau,K_V\setminus Q)\Big)&\leq
\prod_{\{\ell\mid\; l_\ell\in L\setminus Q\}} x_{w_l,s_l}^{(l_\ell)}\,\prod_{\{\ell\mid\; k_\ell\in K\setminus Q\}}x_{v_\ell,\ell}^{(k_\ell)}\,\prod_{\{\ell\mid\; k_\ell\in Q\}}x_{\tau(w'_\ell),d_\ell}^{(q_\ell)},\\\label{inn2}
\inn\Big(d\rr_{\lambda,U}(K_V\setminus Q)D\cc_{b}(Q_\sigma,L_S\setminus Q)\Big)&\leq
\prod_{\{\ell\mid\; k_\ell\in K\setminus Q\}} x_{v_l,u_l}^{(k_\ell)}\,\prod_{\{\ell\mid\; l_\ell\in L\setminus Q\}}x_{\ell,s_\ell}^{(l_\ell)}\,\prod_{\{\ell\mid\; l_\ell\in Q\}}x_{f_\ell,\s(u'_\ell)}^{(q_\ell)},
\end{align}
where $W=\{w_1,\dots,w_{b-c}\}\subset \N_b$, $W^\com=\{w'_1,\dots,w'_c\}$,  $U=\{u_1,\dots,u_{a-c}\}\subset \N_a$, 
$U^\com=\{u'_1,\dots,u'_c\}$, $\tau\in \sym W^\com$, $\s\in \sym U^\com$, 
and we need to exclude $W=\{d_1,\dots,d_c\}^\com$, $\tau=\id$, and $U=\{f_1,\dots,f_c\}^c$, $\s=\id$. 
Note that the equalities hold for $\lambda=b$ (resp. $\lambda=a$).

One easily infers that $\inn(\s_{ij})$ equals the product  
\beq\label{inn}
\prod_{\{\ell\mid\, k_{\ell}\in Q\}}x_{f_\ell,d_\ell}^{(q_\ell)}\prod_{\{\ell\mid\, k_\ell\in K\setminus Q\}}x_{v_\ell,\ell}^{(k_\ell)}\prod_{\{\ell\mid\, l_\ell\in L\setminus Q\}}x_{\ell,s_\ell}^{(l_\ell)},
\eeq
in which we replace the factor $x_{v_{a-1},a-1}^{(k_{a-1})}x_{v_a,a}^{(k_a)}$ (resp. $x_{b-1,s_{b-1}}^{(l_{b-1})}x_{b,s_b}^{(l_b)}$) by  
$x_{v_{a},a-1}^{(k_{a-1})}x_{v_{a-1},a}^{(k_a)}$ (resp.  $x_{b-1,s_{b}}^{(l_{b-1})}x_{b,s_{b-1}}^{(l_b)}$) 
if $k_{a-1}\geq l_{b-1}$ (resp. if $k_{a-1}<l_{b-1}$). 
The terms in \eqref{inn1} are obtained by  permuting the indices corresponding to the rows of the elements $x_{\ell,s_\ell}^{(l_\ell)}$ appearing in 
\eqref{inn}  (notice that  the elements $x_{f_\ell,d_\ell}^{(q_\ell)}$ are also of that form), 
while the terms in  \eqref{inn2} are obtained by  permuting the indices corresponding to the columns of the elements $x_{v_\ell,\ell}^{(k_\ell)}$ appearing in 
\eqref{inn}  (notice that  the elements $x_{f_\ell,d_\ell}^{(q_\ell)}$ are also of that form). 
Since the permutation described in order to obtain the initial term of $\s_{ij}$ leads to the biggest monomial among the monomials in \eqref{inn1}, \eqref{inn2} with respect to the given order $>$ on $C$, $\s_{ij}$ has a standard expression, which concludes the proof.
\end{proof}

We will need a slight generalization of  Lemma \ref{grob}. Let $K=\{k_1,\dots,k_a\}\subset \N_m$, 
$L=\{l_1,\dots,l_b\}\subset \N_m$. We denote  
$$\Xi^{(KL)}=(X'_{k_1},\dots,X'_{k_{a}},X''_{l_1},\dots,X''_{l_{b}}),$$ 
and write $M^{(KL)}$ for the module generated by the rows of $\Xi^{(KL)}$. 
Let $G'^{(K)}\subseteq G'$ be a multilinear Gr\"obner basis on the rows of $(X'_{k_1},\dots,X'_{k_a})$, and 
$G''^{(L)}\subseteq G''$ be a multilinear Gr\"obner basis on the rows of $(X''_{l_1},\dots,X''_{l_b})$. 
We state the next lemma without proof since one only needs to inspect the proof of Lemma \ref{grob}, and notice that it carries over to a more general situation of the following lemma. 
\begin{lemma}\label{grobmod}
The set $G'^{(K)}\cup G''^{(L)}$ is a multilinear Gr\"obner basis of $M^{(KL)}$.
\end{lemma}

\subsection{Reduction to standard and one-sided identities}
We are now in a position to establish our main result on two-sided functional identities, which together with Theorem \ref{FIGPI} gives a full description of functional identities on $M_n(\F)$.

\begin{theorem}\label{2fi}
Let $m\ge 2$ and $K,L\subseteq \N_m$. Every solution of the functional identity
\begin{equation*} 
\sum_{k\in K}F_k(\overline{x}_m^k)x_k = \sum_{l\in L}x_lG_l(\overline{x}_m^l)
\end{equation*} on $M_n(\F)$ is standard modulo one-sided identities.
\end{theorem}

\begin{proof}
The functional identity in question  can be written coordinate-wise as a system of equations. 
We will treat the situation where $K=L=\N_m$. 
Finding all solutions of our functional identity in this case, we will obtain the desired ones among those with $F_k=0$ for $k\in \N_m\setminus K$, $G_l=0$ for $l\in \N_m\setminus L$.
For this let us first denote 
$$
F'_k=(F_{11}^{(k)},\dots,F_{1n}^{(k)},\dots,F_{n1}^{(k)},\dots,F_{nn}^{(k)}),\quad
G''_l=(G_{11}^{(l)},\dots,G_{1n}^{(l)},\dots,G_{n1}^{(l)},\dots,G_{nn}^{(l)}),$$
$$H=\big(F'_1,\dots,F'_m,G''_1,\dots,G''_m\big).$$
Then the system of equations reads as 
\beq\label{sys}
H\Xi=\big(F'_1,\dots,F'_m,G''_1,\dots,G''_m\big)
\left(\begin{array}{c}
X'_1\\
\vdots\\
X'_m\\
X''_1\\
\vdots\\
X''_m
\end{array}
\right)=0
\eeq
This system can thus be interpreted as a syzygy on the rows of the matrix $\Xi$.

In our case $F_k,G_l$ are multilinear so we can restrict ourselves to the syzygies on the rows of $\Xi$ which yield multilinear functions in $x_1,\dots,x_m$. 
These are generated by $\tau_{ij}$ for $i,j$ such that $g_i,g_j$ belong to a multilinear Gr\"obner basis $G$ of $M$. 
By Lemma \ref{grob} we have $G=G'\cup G''$. 
If we take elements $g_i,g_j\in G'$ (resp. $g_i,g_j\in G''$), 
then $\s_{ij}=m_{ji}g_i-m_{ij}g_j$ can be expressed in terms of $g_\ell\in G'$ (resp. $g_\ell\in G''$), 
since $G'$ (resp. $G''$) is a (multilinear) Gr\"obner basis of the module generated by the first (resp. last) $mn^2$ rows of $\Xi$. 
These $\tau_{ij}$ thus yield the one-sided functional identities.
It remains to consider $\tau_{ij}$ for $g_i\in G'$, $g_j\in G''$. 
Both elements $g_i,g_j$ are multilinear of degree at most $n$. 
Let $g_i$ involve the variables appearing in $X_k$ for $k\in K'\subset \N_m$, $|K'|\leq n$, 
and $g_j$ those appearing in $X_l$ for $l\in L'\subset \N_m$, $|L'|\leq n$. 
We can treat $g_i,g_j$ as the elements of the Gr\"obner basis on the rows of the matrix $\Xi^{(K'L')}$. 
By Lemma \ref{grobmod}, $\s_{ij}$ can be expressed in terms of those $g_\ell$ that form a Gr\"obner basis on the $\Xi^{(K'L')}$, 
which implies that the functional identity of the form 
$$\sum_{k\in K'}F_k(\overline{x}_m^k)x_k = \sum_{l\in L'}x_lG_l(\overline{x}_m^l)$$ 
 corresponds to the syzygy $\tau_{ij}$. 
Since $|K'|,|L'|\leq n$, this  identity has standard solutions by  \cite[Corollary 2.23]{FIbook}.
\end{proof}

\subsection{An application: Commuting traces} Let $A$ be an algebra over a field $\F$, and let $T:A\to A$ be the trace of an $r$-linear function $t:A^r\to A$  (see Subsection \ref{subdet}).
 We say that $T$ is {\em commuting} if $[T(x),x] =0$ for all $x\in A$. Such a map is said to be of  a {\em standard form}
if there exist  traces of  $(r-i)$-linear functions $\mu_i:A\to \F$, $0\le i\le r$, such that
  $T(x) = \sum_{i=0}^{r} \mu_i(x)x^i$ for all $x\in A$. The question whether every commuting trace of an $r$-linear function on $A$ is of a standard form has been studied extensively for different classes of algebras $A$ (see \cite{Bre04, FIbook} for surveys), but, paradoxically,
for  the basic case where $A=M_n(\F)$ this question had been opened for a long time and has been answered in affirmative only very recently \cite{BS14}, and only under the assumption that char$(\F)=0$. Our goal now is to show that Theorem \ref{2fi} can be used for obtaining an alternative, independent proof, and moreover, a slight generalization with respect to restrictions on char$(F)$.

\begin{corollary}\label{co1}
Let $T:M_n(\F)\to M_n(\F)$ be a commuting trace of an $r$-linear function $t$.  If  {\rm char}$(\F)=0$ or {\rm char}$(\F) >r+1$, then $T$ is of a standard form.
\end{corollary}

\begin{proof}
The proof is by induction on $r$. The  $r=0$ case is trivial, so we may assume that $r> 0$ and that the result holds for $r-1$. 

Note that the complete linearization of $[T(x),x]=0$  yields the functional identity
\begin{equation*}
\sum_{k=1}^{m}F(\overline{x}_m^k)x_k = \sum_{l=1}^{m}x_lF(\overline{x}_m^l)\quad\mbox{for all $x_1\dots,x_{m}\in M_n(\F)$,}
\end{equation*} 
where $m=r+1$ and 
\begin{equation*} 
F(x_1,\dots,x_r) = \sum_{\sigma\in S_r} t(x_{\sigma(1)},\dots,x_{\sigma(r)}).
\end{equation*} 
Applying Theorem \ref{2fi}  we see that for each
$k=1,\dots,m=r+1$ there exist functions \begin{align*}
&\varphi_k:M_n(\F)^{m-1}\to M_n(\F), \,\,k\in \N_m,\\
&p_{kl}:M_n(\F)^{m-2}\to M_n(\F), \,\,k,l\in \N_m,\,\,k\not=l,\\
&\lambda_k:M_n(\F)^{m-1}\to \F,\,\,k\in \N_m,
\end{align*}
such that 
\begin{equation}\label{enacba}\sum_{k=1}^m\varphi_k(\overline{x}_m^k)x_k =0\end{equation}
 and
$$
F(\ol{x}_m^k)  =   \displaystyle \sum_{l\in \N_m,\,l\not=k}x_lp_{kl}(\ol{x}_m^{kl})
+\lambda_k(\ol{x}_m^k) + \varphi_k(\ol{x}_m^k).
$$
Setting $x_1=\dots=x_m = x$ and using $F(x,\dots,x) = r! T(x)$ we arrive at
\begin{equation} \label{equ1}
T(x)=xP_k(x)+\Lambda_k(x)+\Phi_k(x), \quad 1\leq k\leq m,
\end{equation}
 where  $P_k$ is the trace of $\frac{1}{r!}\sum_{l\in \N_m,\,l\not=k} p_{kl}$, $\Lambda_k$ is the trace of $\frac{1}{r!} \lambda_{k}$, and $\Phi_k$ is the trace of $\frac{1}{r!} \varphi_{k}$. Note that  \eqref{enacba} yields $\sum_{k=1}^m \Phi_k(x)x=0$ for all $x\in M_n(\F)$. Interpreting this identity as $AX=0$ where $X$ is a generic matrix and $A$ is the matrix corresponding to $\sum_{k=1}^m \Phi_k$, it follows (since $X$ is invertible) that $A =0$, and hence that 
 \begin{equation} \label{equ2} \sum_{k=1}^m \Phi_k=0.\end{equation}
 Further, \eqref{equ1} shows that
$$
\Phi_1(x)-\Phi_\ell(x)=x\bigl(P_\ell(x)-P_1(x)\bigr)+\Lambda_\ell(x)-\Lambda_1(x),\quad 2\leq \ell\leq m.$$
Summing up these identities and using \eqref{equ2} we get
\begin{equation}\label{m3}
m\Phi_1(x)=x\Bigl(\sum_{\ell=2}^{m}P_\ell(x)-(m-1)P_1(x)\Bigr)+\sum_{\ell=2}^{m}\Lambda_\ell(x)-(m-1)\Lambda_1(x).
\end{equation}
Accordingly,
$$
\Phi_1(x)-x\tilde P_1(x)\in\F \quad\mbox{for all $x\in M_n(\F)$,}
$$
where $\tilde{P}_1= \frac{1}{m}\Bigl(\sum_{\ell=2}^{m}P_\ell-(m-1)P_1\Bigr)$. Setting $P=P_1 + \tilde{P}_1$ we thus see from \eqref{equ1} that
\begin{equation} \label{equ3}
T(x)-x P(x)\in\F \quad\mbox{for all $x\in M_n(\F)$.}
\end{equation}
Since $T$ is commuting it follows that $[xP(x),x]=0$ for every $x\in M_n(\F)$, which can be written as $x[P(x),x]=0$. Interpreting this identity through a generic matrix we see, similarly as above, that 
$[P(x),x]=0$.
Thus, $P$ is a commuting trace of an $(r-1)$-linear function. By induction assumption, $P$ is of a standard form. From \eqref{equ3} we thus see that  $T$ is of a standard form, too.
\end{proof}

A more careful analysis is needed if one wishes to  further  weaken the assumption on char$(\F)$. Let us examine only the case where $r=2$, which is the one that plays the most prominent role in applications of functional identities. It naturally appears in the study of Lie isomorphisms, commutativity preserving maps, Lie-admissible maps, Poisson algebras, and several other topics  (see \cite{Bre04} and \cite[Section 1.4]{FIbook}). Any new information on commuting traces of bilinear functions is therefore of interest.

We will thus consider the case where $T$ is a commuting trace of a bilinear map $F$. We have to add the assumption that 
$F$ satisfies the functional identity
\begin{equation} \label{eqB}
F(x,y)z + F(z,x)y + F(y,z)x = zF(x,y) + yF(z,x) + xF(y,z).
\end{equation} 
One way of looking at \eqref{eqB} is that $(x,y)\mapsto F(x,y)$ is a nonassociative commutative product on $M_n(F)$ which is connected with the ordinary product through a version of the Jacobi identity:
$[F(x,y),z]  + [F(z,x),y] + [F(y,z),x]=0$.

Note that \eqref{eqB} is actually equivalent to $T$ being commuting provided that $F$ is symmetric and char$(F)\ne 2,3$. However, we do not wish to impose these assumptions. The point of the next corollary is that it  has no restrictions on char$(F)$.



\begin{corollary}
Let $T$ be the trace of a    bilinear function $F:M_n(\F)^2\to M_n(\F)$. If $T$ is commuting and $F$ satisfies \eqref{eqB}, then $T$ is of a standard form.
\end{corollary}

\begin{proof}
If char$(\F)\ne 2$, then the result follows from \cite[Theorem 5.32 and Remark 5.33]{FIbook}. Assume, therefore, that char$(\F)=2$. From now on we simply follow the proof of Corollary  \ref{co1}.
 Thus, we first derive  \eqref{equ1} with $m=3$, and after that \eqref{equ2} and  \eqref{m3}. Note that $m\Phi_1(x) =\Phi_1(x)$ since $m=3$. Therefore, \eqref{equ3} follows with $P$ being a linear function. This implies that $P$ is commuting, and hence is of a standard form \cite[Corollary 5.28]{FIbook}. Consequently, $T$ is of a standard form as well.
\end{proof}

These corollaries can be extended, by using scalar extensions and other results on functional identities, to considerably more general algebras than $M_n(\F)$ (cf. \cite{BS14}). Also, some other 
special identities that have proved to be useful for applications could  now be examined in greater detail. However, 
 we do not wish to make this section lengthy and technical. Our aim has been just to give an indication of the applicability of Theorem \ref{2fi}.

\end{document}